\documentclass[sn-mathphys-num]{sn-jnl}

\usepackage{graphicx}%
\usepackage{multirow}%
\usepackage{amsmath,amssymb,amsfonts}%
\usepackage{amsthm}%
\usepackage{mathrsfs}%
\usepackage[title]{appendix}%
\usepackage{xcolor}%
\usepackage{textcomp}%
\usepackage{manyfoot}%
\usepackage{algorithm}%
\usepackage{algorithmic}%
\usepackage{listings}%
\usepackage{epstopdf}
\usepackage{subfigure}
\usepackage{float}
\usepackage{booktabs}
\usepackage{multirow}
\usepackage{tabularx} 
\usepackage{ragged2e} 
\usepackage{booktabs} 
\usepackage{rotating}
\usepackage{enumitem}

\theoremstyle{thmstyleone}%
\newtheorem{theorem}{Theorem}[section]
\newtheorem{remark}{Remark}[section]
\newtheorem{lemma}{Lemma}[section]
\newtheorem{definition}{Definition}[section]

\numberwithin{equation}{section}

\begin{document}

\title[Convergence Analysis of the Self-Adaptive Projection Method for Variational Inequalities with Non-Lipschitz Continuous Operators]{Convergence Analysis of the Self-Adaptive Projection Method for Variational Inequalities with Non-Lipschitz Continuous Operators}

\author[1]{\fnm{Meiying} \sur{Wang}}\email{xdwangmeiying@163.com}

\author*[1]{\fnm{Hongwei} \sur{Liu}}\email{hwliuxidian@163.com}

\author[2]{\fnm{Jun} \sur{Yang}}\email{xysyyangjun@163.com}

\affil[1]{\orgdiv{School of Mathematics and Statistics}, \orgname{Xidian University}, \orgaddress{\city{Xi'an}, \postcode{710126}, \state{Shaanxi}, \country{China}}}


\affil[2]{\orgdiv{School of Mathematics and Statistics}, \orgname{Xianyang Normal University}, \orgaddress{\city{Xianyang}, \postcode{712000}, \state{Shaanxi}, \country{China}}}


\abstract{In this paper, we employ Tseng's extragradient method with the self-adaptive stepsize to solve variational inequality problems involving non-Lipschitz continuous and quasimonotone operators in real Hilbert spaces. The convergence of the proposed method is analyzed under some mild assumptions. The key advantages of the method are that it does not require the operator associated with the variational inequality to be Lipschitz continuous and that it adopts the self-adaptive stepsize. Numerical experiments are also provided to illustrate the effectiveness and superiority of the method.
}



\maketitle

\section{Introduction}\label{sec1}
This paper focuses on the classical variational inequality associated with the operator $F: H\to H$, denoted by $VI(C, F)$, which is described as finding a point $x^\ell \in C$ such that
\begin{equation*}
  \langle F(x^\ell), w - x^\ell \rangle \geq 0, \ \forall w\in C,
\end{equation*}
where $C$ is a nonempty, closed, and convex subset of a real Hilbert space $H$. \\
Let $S$ be the solution set of $VI(C, F)$, and define $S_D$ as the set of solutions to the dual variational inequality: find $z^\ell \in C$ such that
\begin{equation*}
  \langle F(v), v - z^\ell \rangle \geq 0, \ \forall v \in C.
\end{equation*}
It is evident that $S_D$ forms a closed and convex set, although it may be empty. When $F$ is continuous and $C$ is convex, the inclusion $S_D \subseteq S$ holds. Furthermore, if $F$ is both pseudomonotone and continuous, then $S_D = S$ \cite{co}. However, it should be noted that the converse inclusion $S \subseteq S_D$ does not necessarily hold when $F$ is quasimonotone and continuous \cite{ye}.

Variational inequalities can effectively model a broad class of equilibrium and optimization problems, and have consequently become a central framework in both theoretical and applied disciplines. Due to their extensive applications in areas such as economics, transportation systems, and engineering mechanics \cite{kinderlehrer,konnov,vip}, a wide range of methods have been developed and investigated for solving $VI(C, F)$, such as \cite{ye,kor,sol,cen,tseng2000,mali,dong,yang2018,yang3,shehu,yao,Tan2023,shehu2019, yang2019strong}.

One well-known method for solving variational inequalities in finite dimensional spaces is the extragradient algorithm, originally introduced by Korpelevich \cite{kor} and independently by Antipin \cite{an}. This method applies to cases where $F$ satisfies monotonicity and Lipschitz conditions, and each iteration involves performing two projections onto the constraint set $C$. However, computing these projections can be challenging or infeasible in practice when $C$ has a complex structure or lacks an explicit projection formula. Censor et al. \cite{cen} proposed an alternative to the classical extragradient scheme, known as the subgradient extragradient method, where the usual second projection onto the feasible set $C$ is replaced by one onto a dynamically generated half space $T_n$.
This approach retains convergence guarantees while reducing computational complexity. The iterative steps are given by:
\begin{equation}\label{1.3}
\begin{cases}
w_n = P_C(z_n - \lambda F(z_n)), \\
T_n = \{v \in H : \langle z_n - \lambda F(z_n) - w_n, v - w_n \rangle \leq 0\},\\
z_{n+1} = P_{T_n}(z_n - \lambda F(w_n)),
\end{cases}
\end{equation}
where the stepsize parameter $\lambda$ satisfies $0 < \lambda < \frac{1}{L}$, with $L$ denoting the Lipschitz constant of the monotone operator $F$.\\
An alternative strategy was introduced by Tseng in \cite{tseng2000}. Instead of using the second projection, this method updates the iterate through an explicit correction step, described as:
\begin{equation}\label{1.4}
\begin{cases}
w_n = P_C(z_n - \lambda F(z_n)), \\
z_{n+1} = w_n - \lambda (F(w_n) - F(z_n)),
\end{cases}
\end{equation}
with the same condition $\lambda \in (0, \frac{1}{L})$.

The step sizes used in (\ref{1.3}) and (\ref{1.4}) depend on the Lipschitz constant of the operator $F$, which is often unknown or hard to compute in practice. To address this problem, Liu and Yang \cite{yang3} introduced a modification of algorithm (\ref{1.4}) by incorporating a non-monotonic self-adaptive stepsize strategy:
\begin{equation}\label{liuyang}
\lambda_{n+1}=
\begin{cases}
\min\{\frac{\mu\|z_n-w_n\|}{\|F(z_n)-F(w_n)\|},\ \lambda_n+\xi_n\},\ \ \text{if}\ F(z_n)-F(w_n)\neq 0,\\
\lambda_n+\xi_n,\qquad\qquad\qquad\qquad\qquad\ \text{otherwise},
\end{cases}
\end{equation}
where $F$ is Lipschitz continuous operator, $\mu\in(0,1)$ and $\{\xi_n\}\subset [0,\infty)$ such that $\sum\limits_{n=1}^{\infty}\xi_n<\infty$.\\
It is well known that methods (\ref{1.3}), (\ref{1.4}), and (\ref{liuyang}) all require the operator $F$ to be Lipschitz continuous. However, when $F$ fails to satisfy this condition, these methods may no longer be applicable. To overcome this limitation, Iusem \cite{Iusem} proposed a linesearch-based approach to solve variational inequalities in finite dimensional spaces, where the operator $F$ is assumed to be monotone and continuous. This algorithm is described as follows:
\begin{equation}\label{1.5}
\begin{cases}
w_n = P_C(z_n - \eta_n F(z_n)), \\
z_{n+1} = P_C(z_n - \lambda_n F(w_n)),
\end{cases}
\end{equation}
where $\eta_n:= \eta l^{m_n}$, $\eta>0$ and $m_n$ is the smallest non-negative integer $m$ satisfying for $\mu, l\in(0,1)$,
\begin{equation*}\label{1.6}
 \eta l^{m}\|F(z_n)-F(w_n)\|\leq \mu\|z_n-w_n\|,
\end{equation*}
and
\begin{equation*}
\lambda_{n}:=\frac{\langle F(w_n), z_n-w_n\rangle}{\|F(w_n)\|^2}.
\end{equation*}
The author conducted the convergence analysis in finite dimensional space. Subsequently, various extensions have been investigated by many researchers; see, for example, \cite{thong201911, thong2020weak, cai2021strong, xie2025modified}.
Among them, Thong et al. \cite{thong201911} employed the Tseng's extragradient algorithm combined with the linesearch technique to solve variational inequalities in real Hilbert spaces, where the operator is pseudomonotone and non-Lipschitz continuous. The detailed iteration schemes are as follows:
\begin{equation}\label{1.7}
\begin{cases}
w_n = P_C(z_n - \lambda_n F(z_n)), \\
z_{n+1} = w_n - \lambda_n(F(w_n) - F(z_n)),
\end{cases}
\end{equation}
where $\lambda_n:= \lambda l^{m_n}$, $\lambda>0$ and $m_n$ is the smallest non-negative integer $m$ satisfying for $\mu, l\in(0,1)$,
\begin{equation*}\label{1.8}
 \lambda l^{m}\|F(z_n)-F(w_n)\|\leq \mu\|z_n-w_n\|.
\end{equation*}
Although methods (\ref{1.5}) and (\ref{1.7}) are effective for solving variational inequalities with non-Lipschitz continuous operators, it is evident that both methods rely on linesearch techniques to determine the stepsizes, including the algorithms in \cite{ thong2020weak, cai2021strong, xie2025modified}. However, a major drawback of this approach is that once the linesearch procedure is initiated, multiple projections and function value computations are required during the calculation, which can significantly increase the computational cost. Therefore, adopting the self-adaptive stepsize method becomes both necessary and advantageous.

The objective of this paper is to solve variational inequalities involving quasimonotone and non-Lipschitz continuous operators in real Hilbert spaces by employing Tseng's extragradient method combined with the self-adaptive stepsize. Our analysis demonstrates the convergence of the sequences generated by the algorithm in finite dimensions, and also shows strong convergence in infinite dimensional Hilbert spaces under appropriate conditions.

The subsequent sections of this paper are organized as follows: Section \ref{sec2} provides the necessary definitions and foundational results that will be used throughout the analysis. In Section \ref{sec3}, we propose the algorithm and the associated assumptions for solving variational inequalities with quasimonotone and non-Lipschitz continuous operators. Section \ref{sec4} is devoted to the convergence of the established methods under some assumptions. Finally, in Section \ref{sec55}, we validate the performance and efficiency of the algorithm through numerical experiments.

\section{Preliminaries}\label{sec2}
This section presents several fundamental concepts and lemmas that play a crucial role in our subsequent analysis. Throughout the paper, the notation $ ``\rightharpoonup"$ indicates weak convergence, while $``\to"$ denotes strong convergence.
\begin{definition}
An operator $F: H\to H$ is said to be
\\ \emph{(i)} monotone if
\begin{equation*}
  \langle F(w)-F(z), w-z \rangle\geq0,\quad \forall w,\ z\in H.
\end{equation*}
\\ \emph{(ii)} $\eta$-strongly pseudomonotone if there exists a positive constant $\eta$ such that
\begin{equation*}
  \langle F(w), z-w \rangle\geq 0\Longrightarrow \langle F(z), z-w \rangle\geq\eta\|w-z\|^2,\quad \forall w,\ z\in H.
\end{equation*}
\\ \emph{(iii)} pseudomonotone if
\begin{equation*}
  \langle F(w), z-w \rangle\geq0\Longrightarrow \langle F(z), z-w \rangle\geq0,\quad \forall w,\ z\in H.
\end{equation*}
\\ \emph{(iv)} quasimonotone if
\begin{equation*}
  \langle F(w), z-w \rangle>0\Longrightarrow \langle F(z), z-w \rangle\geq0,\quad \forall w,\ z\in H.
\end{equation*}
\\ \emph{(v)} uniformly continuous if for any $\epsilon>0$, there exists $\delta>0$ such that
\begin{equation*}
  \|w-z\|< \delta\Longrightarrow \|F(w)-F(z)\|<\epsilon,\quad \forall w,\ z\in H.
\end{equation*}
\\ \emph{(vi)} $L$-Lipschitz continuous with $L>0$ if
\begin{equation*}
  \|F(w)-F(z)\|\leq L\|w-z\|,\quad \forall w,\ z\in H.
\end{equation*}
\end{definition}
Obviously, (i)$\Rightarrow$(iii)$\Rightarrow$(iv), (ii)$\Rightarrow$(iii)$\Rightarrow$(iv) and (vi)$\Rightarrow$(v). However, the converse do not necessarily hold true.
\begin{lemma}\cite{vanderbei1991}\label{1991}
Let $C$ be a convex subset of $H$. Then, $F: C\to H$ is uniformly continuous if and only if, for any $\epsilon > 0$, there exists a positive constant $M < +\infty$ such that for all $w, z\in C$,
\[
\|F(z) - F(w)\| \leq M\|z - w\| + \epsilon.
\]
\end{lemma}
Given a point $u$, its metric projection onto the set $C$, denoted by $P_C(u)$, and it is defined as:
\begin{equation*}
  P_C(u):=\mathop{\arg\min}\{\|u-v\|:v\in C\}.
\end{equation*}
\begin{lemma}\label{lm:2.1}\cite{bau}
The following results hold for any $z\in H$ and any $u\in C$:
\\ \emph{(i)} $\langle z-P_C(z),\ u-P_C(z)\rangle\leq0$.
\\ \emph{(ii)} $\|z-P_C(z)\|^2\leq \langle z-P_C(z),\ z-u \rangle$.
\end{lemma}

\begin{lemma}\label{mono}\cite{denisov2015}
Let $z \in H$ and let $\alpha,\ \beta > 0$ with $\alpha \geq \beta$. Define $z(\alpha)=P_C(z-\alpha F(z))$ and $e(z,\alpha)=z-z(\alpha)$. Then the following inequalities are satisfied:
\begin{equation*}
\frac{1}{\alpha} \left\|e(z,\alpha)\right\| \leq \frac{1}{\beta} \left\|e(z,\beta)\right\|,
\end{equation*}
and
\begin{equation*}
\left\|e(z,\beta)\right\| \leq \left\|e(z,\alpha)\right\|.
\end{equation*}
\end{lemma}

\begin{lemma}\cite{polyak1987}\label{lim}
Let $\{u_n\}$, $\{\beta_n\}$ and $\{v_n\}$ be sequences of nonnegative real numbers such that
\begin{equation*}
 u_{n+1}=(1 + \beta_n)u_n + v_n, \ \forall n \geq 1,
\end{equation*}
where $\sum_{n=0}^\infty \beta_n<+\infty$ and $\sum_{n=0}^\infty v_n<+\infty$. Then, $\lim\limits_{n\to\infty} u_n$ exists.
\end{lemma}

\section{Proposed Method}\label{sec3}
In this section, we first present the assumptions throughout the paper as follows:
\begin{description}
 \item[(A1)]$S_D\neq \emptyset$.
 \item[(A2)]$F: H\to H$ is quasimonotone operator.
 \item[(A3)]$F$ is uniformly continuous operator on $H$.
 \item[(A4)]The operator $F$ satisfies whenever $\{z_n\}\subset H$ and $z_n\rightharpoonup z$, there is $\|F(z)\|\leq \liminf\limits_{n\to\infty}\|F(z_n)\|$.
\end{description}
The following is the self-adaptive Tseng's extragradient algorithm.
\renewcommand{\thealgorithm}{3.1}
\begin{algorithm}[H]
\caption{}
\label{AL1}
\begin{algorithmic}
\STATE{\emph{Step 0:} Let $\epsilon\geq0$, $\mu\in(0,1)$ and $\{\xi_n\}\subset [0,\infty)$ such that $\sum\limits_{n=1}^{\infty}\xi_n<\infty$. Choose arbitrarily initial point $z_1\in H$ and $\lambda_1>0$.}
\STATE{\emph{Step 1:} Compute
\begin{equation*}
w_n=P_C(z_n-\lambda_nF(z_n)).
\end{equation*}
If $\|z_n-w_n\|\leq\lambda_n\epsilon$, then terminate the process ($z_n$ is an approximate solution). Otherwise,
}
\STATE{\emph{Step 2:} Compute
\begin{equation*}
  z_{n+1}=w_n+\lambda_n(F(z_n)-F(w_n)).
\end{equation*}
\STATE{\emph{Step 3:} Update
\begin{equation*}
\lambda_{n+1}=
\begin{cases}
\min\{\frac{\mu\|z_n-w_n\|}{\|F(z_n)-F(w_n)\|},\ \lambda_n+\xi_n\},\ \ \text{if}\ F(z_n)-F(w_n)\neq 0,\\
\lambda_n+\xi_n,\qquad\qquad\qquad\qquad\qquad\ \text{otherwise}.
\end{cases}
\end{equation*}
}
}
\STATE{\emph{Step 4:}\ Set $n:= n + 1$ to update $n$, and then return to
\emph{Step 1}.}
\end{algorithmic}
\end{algorithm}

\begin{remark}
Under the assumption that $F$ satisfies the Lipschitz condition, the convergence analysis of Algorithm \ref{AL1} follows the same framework as established in \cite{yang3}. This work focuses on the case where $F$ is uniformly continuous.

\end{remark}

\section{Convergence Analysis}\label{sec4}
In this section, we assume $\epsilon=0$ and $z_n\neq w_n$ for all $n$, which implies the algorithm does not terminate in finite steps. Subsequently, we focus on establishing the convergence properties of the infinite sequence $\{z_n\}$ obtained by Algorithm \ref{AL1}.


\begin{lemma}\label{lll}
The sequence $\{\lambda_n\}$ is generated by Algorithm \ref{AL1}. Then,\\
\emph{(i)} $\{\lambda_n\}$ has upper bound $\lambda_1+\Xi$, where $\Xi=\Sigma_{n=1}^\infty \xi_n$;\\
\emph{(ii)} The limit of $\{\lambda_n\}$ exists, that is, $\lim\limits_{n\to\infty}\lambda_n=\lambda\geq0$;\\
\emph{(iii)} The series $\sum_{n=0}^{\infty}(\lambda_{n+1}-\lambda_n)_+$ and $\sum_{n=0}^{\infty}(\lambda_{n+1}-\lambda_n)_-$ are convergent, where $(\lambda_{n+1}-\lambda_n)_+=\max\{0, \lambda_{n+1}-\lambda_n\}$ and $(\lambda_{n+1}-\lambda_n)_-=\max\{0, -(\lambda_{n+1}-\lambda_n)\}$.
\end{lemma}

\begin{proof}
A detailed proof can be obtained from Lemma 3.1 in \cite{yang3}.
\end{proof}

\begin{lemma}\label{22}
Assume that conditions (A1) and (A2) hold. Then, the sequence $\{z_n\}$ generated by Algorithm \ref{AL1} satisfies \\
\emph{(i)} $\lim_{n\to\infty} \|z_n - x^\ell\|$ exists for $x^\ell\in S_D$;\\
\emph{(ii)} $\lim_{n\to\infty}\|z_{n}-w_n\|= 0$ and $\lim_{n\to\infty}\|z_{n+1}-z_n\|=0$.
\end{lemma}
\begin{proof}
Let $x^\ell\in S_D$. Then, from $z_{n+1}=w_n-\lambda_n(F(w_n)-F(z_n))$, we have
\begin{align}\label{3.1}
  \|z_{n+1}-x^\ell\|^2=&\|w_n-\lambda_n(F(w_n)-F(z_n))-x^\ell\|^2\notag\\
  =&\|w_n-x^\ell\|^2+\lambda_n^2\|F(w_n)-F(z_n)\|^2-2\lambda_n\langle F(w_n)-F(z_n), w_n-x^\ell \rangle\notag\\
  =&\|z_n-x^\ell\|^2+\|w_n-z_n\|^2+2\langle w_n-z_n, z_n-x^\ell\rangle+\lambda_n^2\|F(w_n)-F(z_n)\|^2\notag\\
  &-2\lambda_n\langle F(w_n)-F(z_n), w_n-x^\ell\rangle\notag\\
  =&\|z_n-x^\ell\|^2+\|w_n-z_n\|^2-2\langle w_n-z_n, w_n-z_n\rangle+2\langle w_n-z_n, w_n-x^\ell\rangle\notag\\
  &+\lambda_n^2\|F(w_n)-F(z_n)\|^2-2\lambda_n\langle F(w_n)-F(z_n), w_n-x^\ell\rangle.
\end{align}
According to $w_n = P_C(z_n-\lambda_nF(z_n))$ and $x^\ell\in S_D\subseteq C$, then from Lemma \ref{lm:2.1} (i), we obtain
\begin{equation*}
  \langle w_n-z_n+\lambda_nF(z_n),  w_n-x^\ell\rangle\leq 0,
\end{equation*}
that is,
\begin{equation}\label{3.2}
  \langle w_n-z_n,  w_n-x^\ell\rangle\leq -\lambda_n\langle F(z_n), w_n-x^\ell\rangle.
\end{equation}
Therefore,
\begin{align}\label{3.331}
  \|z_{n+1}-x^\ell\|^2\leq&\|z_n-x^\ell\|^2-\|w_n-z_n\|^2+\lambda_n^2\|F(w_n)-F(z_n)\|^2
  -2\lambda_n\langle F(w_n), w_n-x^\ell \rangle.
\end{align}
Since $x^\ell\in S_D$ and $w_n\in C$, we can obtain $\langle F(w_n), w_n-x^\ell\rangle\geq 0$. Thus, by (\ref{3.331}), we have
\begin{align}\label{3.31}
  \|z_{n+1}-x^\ell\|^2
  \leq&\|z_n-x^\ell\|^2-\|w_n-z_n\|^2+\lambda_n^2\|F(w_n)-F(z_n)\|^2.
\end{align}
For clarity, the proof is organized into two parts based on different cases.\\
\textbf{Case A}
Suppose $\lambda_n\to\lambda>0$ as $n\to\infty$. Then the subsequent proof is the same as that of Lemma 3.2 in \cite{yang3}. Consequently, (i) and (ii) hold.\\
\textbf{Case B}
Suppose $\lambda_n\to0$ as $n\to\infty$. According to Lemma \ref{1991}, by taking $\epsilon=\frac{1}{2}$, we obtain
\begin{equation}\label{dengjia}
  \|F(z_n)-F(w_n)\|\leq M\|z_n-w_n\|+\frac{1}{2}.
\end{equation}
Since $\lim\limits_{n\to\infty}\lambda_n=0$, there exists $N$ such that for all $n>N$, $\lambda_{n}<\frac{\mu}{2M}$.\\
If $\lambda_{n+1}<\lambda_n$, then, by the construction of the sequence $\{\lambda_{n+1}\}$, it follows that
\begin{equation}\label{eq4.5}
  \lambda_{n+1}=\frac{\mu\|z_n-w_n\|}{\|F(z_n)-F(w_n)\|}.
\end{equation}
Therefore, by (\ref{dengjia}), we have
\begin{equation}\label{eq222}
  \|F(z_n)-F(w_n)\|\leq \frac{M}{\mu}\lambda_{n+1}\|F(z_n)-F(w_n)\|+\frac{1}{2}.
\end{equation}
In addition, $\lambda_{n+1} < \lambda_n<\frac{\mu}{2M}$. Thus, by (\ref{eq222}) we have
\begin{equation*}
  \|F(z_n)-F(w_n)\|\leq 1.
\end{equation*}
Hence, from (\ref{3.31}), (\ref{eq4.5}) and $\mu\in(0,1)$, we obtain
\begin{align}\label{bb}
  \|z_{n+1}-x^\ell\|^2\leq&\|z_n-x^\ell\|^2-
  \|z_n-w_n\|^2+\lambda_n^2\|F(z_n)-F(w_n)\|^2\notag\\
  \leq&\|z_n-x^\ell\|^2-
  \frac{\lambda_{n+1}^2}{\mu^2}\|F(z_n)-F(w_n)\|^2
  +\lambda_n^2\|F(z_n)-F(w_n)\|^2\notag\\
  =&\|z_n-x^\ell\|^2-(\frac{1}{\mu^2}-1)\lambda_{n+1}^2\|F(z_n)-F(w_n)\|^2
  -\lambda_{n+1}^2\|F(z_n)-F(w_n)\|^2\notag\\
  &+\lambda_n^2\|F(z_n)-F(w_n)\|^2\notag\\
  \leq&\|z_n-x^\ell\|^2-(1-\mu^2)\frac{\lambda_{n+1}^2}{\mu^2}\|F(z_n)-F(w_n)\|^2
  +(\lambda_{n}^2-\lambda_{n+1}^2)\notag\\
  \leq&\|z_n-x^\ell\|^2-(1-\mu^2)\|z_n-w_n\|^2+(\lambda_{n+1}+\lambda_n)(\lambda_{n+1}-\lambda_n)_-\notag\\
  \leq&\|z_n-x^\ell\|^2-(1-\mu^2)\|z_n-w_n\|^2+\frac{\mu}{M}(\lambda_{n+1}-\lambda_n)_-.
\end{align}
If $\lambda_{n+1}\geq\lambda_n$, according to the definition of $\{\lambda_{n+1}\}$, we have
\begin{equation*}
  \lambda_{n+1}\leq\frac{\mu\|z_n-w_n\|}{\|F(z_n)-F(w_n)\|}.
\end{equation*}
Then from (\ref{3.31}), we can obtain
\begin{align}\label{aa}
  \|z_{n+1}-x^\ell\|^2\leq&\|z_n-x^\ell\|^2-\|w_n-z_n\|^2
  +\lambda_n^2\|F(w_n)-F(z_n)\|^2\notag\\
  \leq&\|z_n-x^\ell\|^2-\|w_n-z_n\|^2
  +\frac{\mu^2\lambda_n^2}{\lambda_{n+1}^2}\|w_n-z_n\|^2\notag\\
  =&\|z_n-x^\ell\|^2-
  (1-\frac{\mu^2\lambda_n^2}{\lambda_{n+1}^2})\|w_n-z_n\|^2\notag\\
  \leq&\|z_n-x^\ell\|^2-
  (1-\frac{\mu^2\lambda_n^2}{\lambda_n^2})\|w_n-z_n\|^2\notag\\
  =&\|z_n-x^\ell\|^2-(1-\mu^2)\|w_n-z_n\|^2\notag\\
  \leq&\|z_n-x^\ell\|^2-(1-\mu^2)\|w_n-z_n\|^2+\frac{\mu}{M}(\lambda_{n+1}-\lambda_n)_-.
\end{align}
Therefore, combining the above discussion, we conclude that for all $n>N$,
\begin{align}\label{48}
  \|z_{n+1}-x^\ell\|^2\leq&\|z_n-x^\ell\|^2-(1-\mu^2)\|z_n-w_n\|^2+\frac{\mu}{M}(\lambda_{n+1}-\lambda_n)_-\notag\\
  \leq&\|z_n-x^\ell\|^2+\frac{\mu}{M}(\lambda_{n+1}-\lambda_n)_-.
\end{align}
Then, applying Lemma \ref{lll} (iii) together with Lemma \ref{lim}, we can get $\lim\limits_{n\to\infty}\|z_n-x^\ell\|$ exists. Consequently, the sequence $\{z_n\}$ is bounded. This completes the proof of (i).\\
From (\ref{48}), we obtain for all $n>N$,
\begin{align*}
  (1-\mu^2)\|z_n-w_n\|^2\leq&\|z_n-x^\ell\|^2-\|z_{n+1}-x^\ell\|^2
  +\frac{\mu}{M}(\lambda_{n+1}-\lambda_n)_-.
\end{align*}
Then we can get
\begin{equation}\label{eq4.8}
  (1-\mu^2)\sum_{j=N}^{n}\|z_j-w_j\|^2\leq\|z_{N}-x^\ell\|^2-\|z_{n+1}-x^\ell\|^2
  +\sum_{j=N}^{n}(\lambda_{j+1}-\lambda_j)_-.
\end{equation}
Taking $n\to\infty$ in (\ref{eq4.8}) and by the fact that $\sum_{j=N}^{\infty}(\lambda_{j+1}-\lambda_j)_-$ is convergent, we have $\sum_{j=N}^{\infty}\|z_j-w_j\|^2<+\infty$. Therefore,
\begin{equation}\label{xy0}
  \lim\limits_{n\to\infty}\|z_n-w_n\|=0.
\end{equation}
Furthermore, by Lemma \ref{lll} (i), we derive
\begin{align}\label{xn1}
  \|z_{n+1}-z_n\|\leq&\|w_n+\lambda_n(F(z_n)-F(w_n))-z_n\|\notag\\
  \leq&\|w_n-z_n\|+\lambda_n\|F(z_n)-F(w_n)\|\notag\\
  \leq&\|w_n-z_n\|+(\lambda_1+\Xi)\|F(z_n)-F(w_n)\|.
\end{align}
Then from (\ref{xy0}) and the uniform continuity of $F$, we obtain $\lim\limits_{n\to\infty}\|F(z_n)-F(w_n)\|=0$. Thus, it follows from (\ref{xn1}) that
\begin{equation*}
  \lim\limits_{n\to\infty}\|z_{n+1}-z_n\|=0.
\end{equation*}
This completes the proof of (ii).
\end{proof}

\begin{lemma}\label{444}
Assume that $\lambda_n\to0$ as $n\to\infty$. Let $I := \{n \in \mathbb{N} : \lambda_{n+1} < \lambda_n\}$, and let $\{n_k\}$ for all $k\geq1$ represent all indicators in $I$. Then, the following properties hold for all $n_k$:\\
\emph{(i)} $\lim\limits_{k\to\infty}\frac{\|z_{n_k}-w_{n_k}\|}{\lambda_{n_k}}=0$;\\
\emph{(ii)} $\lim\limits_{k\to\infty}\frac{\|z_{{n_k}+1}-z_{n_k}\|}{\lambda_{n_k}}=0$.
\end{lemma}
\begin{proof}
Since $\lim\limits_{n\to\infty}\lambda_n=0$, we get that $I := \{n \in \mathbb{N} : \lambda_{n+1} < \lambda_n\}$ is an infinite set. Then, for all $n_k\in I$, we have  $\lambda_{{n_k} + 1} < \lambda_{n_k}$. Therefore, according to the definition of step size, we can get
\begin{equation*}
  \lambda_{n_k + 1}=\frac{\mu\|z_{n_k}-w_{n_k}\|}{\|F(z_{n_k})-F(w_{n_k})\|}
  <\lambda_{n_k},
\end{equation*}
that is,
\begin{equation}
  \frac{\|z_{n_k}-w_{n_k}\|}{\lambda_{n_k}}
  \leq\frac{1}{\mu}\|F(z_{n_k})-F(w_{n_k})\|.
\end{equation}
By (\ref{xy0}), we have established that $\lim\limits_{k\to\infty}\|z_{n_k}-w_{n_k}\|=0$. Combining this result with the uniformly continuity of $F$, we immediately obtain
\begin{equation*}
  \lim\limits_{k\to\infty}\|F(z_{n_k})-F(w_{n_k})\|=0.
\end{equation*}
Consequently,
\begin{equation*}
   \frac{\|z_{n_k}-w_{n_k}\|}{\lambda_{n_k}}\to0,\ \text{as}\ k\to\infty.
\end{equation*}
Thus, the proof of (i) is complete.\\
In addition, from (\ref{xn1}), we obtain
\begin{equation*}
  \|z_{n_k+1}-z_{n_k}\|\leq\|z_{n_k}-w_{n_k}\|+\lambda_{n_k}\|F(z_{n_k})-F(w_{n_k})\|.
\end{equation*}
Therefore,
\begin{equation*}
  \frac{\|z_{n_k+1}-z_{n_k}\|}{\lambda_{n_k}}\leq\frac{\|z_{n_k}-w_{n_k}\|}{\lambda_{n_k}}
  +\|F(z_{n_k})-F(w_{n_k})\|\to0,\ \text{as}\ k\to\infty.
\end{equation*}
This completes the proof of (ii).
\end{proof}

\begin{lemma}\label{333}
Assume that the conditions (A1)-(A4) hold. The sequence $\{z_n\}$ is generated by Algorithm \ref{AL1}. Then there exists a weak cluster point $u^*$ of $\{z_n\}$ such that $u^* \in S_D$ or $F(u^*) = 0$.
\end{lemma}
\begin{proof}
If $\lambda_{n}\to\lambda>0$ as $n\to\infty$, then $\lim\limits_{n\to\infty}\lambda_{n_k}\to\lambda>0$. As shown in the proof of Lemma 3.3 in \cite{yang3}, it follows that all weak cluster points $u^*$ of $\{z_n\}$ satisfy $u^*\in S_D$ or $F(u^*)=0$. In the following we consider the case $\lim\limits_{n\to\infty}\lambda_{n}=0$.\\
If $\lambda_{n}\to0$ as $n\to\infty$, then clearly $I:=\{n\in\mathbb{N}: \lambda_{n+1}<\lambda_{n}\}$ is an infinite set. Let us denote the all indices in set $I$ by $\{n_k\}$. The boundedness of $\{z_{n_k}\}$ implies that there exists a subsequence converges weakly to $u^*$. Without loss of generality, we assume that $z_{n_k}\rightharpoonup u^*$ as $k\to\infty$. In fact (\ref{xy0}) implies that $w_{n_k}\rightharpoonup u^*$ and $u^*\in C$.\\
Moreover, by Lemma \ref{444} (i), we can get
\begin{equation}\label{555}
  \frac{\|z_{n_k}-w_{n_k}\|}{\lambda_{n_k}}
  \to0\ \text{as}\ k\to\infty.
\end{equation}
In the following, we discuss the form of weak cluster points in two cases.\\
\textbf{Case I} If $\limsup\limits_{k\to\infty}\|F(w_{n_k})\|=0$, then by $w_{n_k}\rightharpoonup u^*$ and assumption \emph{(A4)}, we have
\begin{equation*}
  0\leq\|F(u^*)\|\leq\liminf\limits_{k\to\infty}\|F(w_{n_k})\|=0.
\end{equation*}
Hence, $F(u^*)=0$.\\
\textbf{Case II} If  $\limsup\limits_{k\to\infty}\|F(w_{n_k})\|>0$, it is not loss of generality to assume that $\lim\limits_{k\to\infty}\|F(w_{n_k})\|=S>0$. Thus there exists $K\in \mathbb{N}$ such that for all $k\geq K$, $\|F(w_{n_k})\|>\frac{S}{2}$. From $w_{n_k}=P_C(z_{n_k}-\lambda_{n_k}F(z_{n_k}))$ and Lemma \ref{lm:2.1} (i), we get
\begin{equation*}
  \langle z_{n_k}-\lambda_{n_k}F(z_{n_k})-w_{n_k}, q-w_{n_k} \rangle\leq0,\ \forall q\in C,
\end{equation*}
i.e.,
\begin{equation*}
  \frac{1}{\lambda_{n_k}}\langle z_{n_k}-w_{n_k}, q-w_{n_k} \rangle\leq\langle F(z_{n_k}), q-w_{n_k}\rangle,\ \forall q\in C.
\end{equation*}
Hence,
\begin{equation}\label{cc}
  \frac{1}{\lambda_{n_k}}\langle z_{n_k}-w_{n_k}, q-w_{n_k} \rangle-\langle F(z_{n_k})-F(w_{n_k}), q-w_{n_k}\rangle\leq\langle F(w_{n_k}), q-w_{n_k}\rangle,\ \forall q\in C.
\end{equation}
According to $\lim\limits_{k\to\infty}\|z_{n_k}-w_{n_k}\|=0$ and $F$ is uniformly continuity, it follows that
\begin{equation*}
  \lim\limits_{k\to\infty}\|F(z_{n_k})-F(w_{n_k})\|=0.
\end{equation*}
Therefore, by (\ref{555}) and taking $k\to \infty$ in (\ref{cc}) we obtain
\begin{equation}\label{2}
  0\leq\liminf\limits_{k\to\infty}\langle F(w_{n_k}), q-w_{n_k}\rangle\leq\limsup\limits_{k\to\infty}\langle F(w_{n_k}), q-w_{n_k}\rangle<+\infty.
\end{equation}
If $\limsup\limits_{k\to\infty}\langle F(w_{n_k}), q-w_{n_k}\rangle>0$, then there exists a subsequence $\{w_{n_{k_j}}\} \subset \{w_{n_k}\}$ such that $\lim\limits_{j\to\infty}\langle F(w_{n_{k_j}}), q-w_{n_{k_j}}\rangle>0$. Therefore there exists $j_0$ such that for all $j\geq j_0$, $\langle F(w_{n_{k_j}}), q-w_{n_{k_j}}\rangle>0$. By the quasimonotonicity of $F$, it follows that $\langle F(q), q - w_{n_{k_j}} \rangle \geq 0$ for every $j \geq j_0$. Taking the limit as $j\to \infty$, we obtain $\langle F(q), q-u^* \rangle\geq 0$, which implies $u^*\in S_D$.\\
If $\limsup\limits_{k\to\infty}\langle F(w_{n_k}), q-w_{n_k}\rangle=0$, from (\ref{2}) we can get
\begin{equation*}
  \lim\limits_{k\to\infty}\langle F(w_{n_k}), q-w_{n_k}\rangle=\limsup\limits_{k\to\infty}\langle F(w_{n_k}), q-w_{n_k}\rangle=\liminf\limits_{k\to\infty}\langle F(w_{n_k}), q-w_{n_k}\rangle=0.
\end{equation*}
We choose a positive decreasing sequence $\{\omega_k\}$ with $\lim\limits_{k\to\infty}\omega_k=0$. Let $\Gamma_k=|\langle F(w_{n_k}), q-w_{n_k} \rangle|+\omega_k>0$ and $\rho_{n_k}=\frac{F(w_{n_k})}{\|F(w_{n_k})\|^2},\ \forall k\geq K$, so we can get $\langle F(w_{n_k}), \rho_{n_k} \rangle=1$. Thus, $\langle F(w_{n_k}), q+\Gamma_k\rho_{n_k}-w_{n_k}\rangle>0,\ \forall k\geq K$. And since $F$ is quasimonotone, there is $\langle F(q+\Gamma_k\rho_{n_k}), q+\Gamma_k\rho_{n_k}-w_{n_k}\rangle\geq0,\ \forall k\geq K$ holds, which implies that for all $k\geq K$,
\begin{align}\label{3}
  \langle F(q), q+\Gamma_k\rho_{n_k}-w_{n_k}\rangle=&\langle F(q)-F(q+\Gamma_k\rho_{n_k}), q+\Gamma_k\rho_{n_k}-w_{n_k}\rangle\notag\\
  &+\langle F(q+\Gamma_k\rho_{n_k}), q+\Gamma_k\rho_{n_k}-w_{n_k}\rangle\notag\\
  \geq&\langle F(q)-F(q+\Gamma_k\rho_{n_k}), q+\Gamma_k\rho_{n_k}-w_{n_k}\rangle\notag\\
  \geq&-\|F(q)-F(q+\Gamma_k\rho_{n_k})\|\|q+\Gamma_k\rho_{n_k}-w_{n_k}\|.
\end{align}
Due to the fact that $\|\rho_{n_k}\|<\frac{2}{S}$, this implies that $\{\rho_{n_k}\}$ is bounded. Moreover, since $\{w_{n_k}\}$ is bounded and $\lim\limits_{k\to\infty}\Gamma_k=0$, then $\{\|q+\Gamma_k\rho_{n_k}-w_{n_k}\|\}$ is bounded. And by the uniformly continuity of $F$, we can get
\begin{equation*}
  \|F(q)-F(q+\Gamma_k\rho_{n_k})\|\to0,\ \text{as}\ k\to\infty.
\end{equation*}
Then, for all $q\in C$, $\langle F(q), q-u^* \rangle\geq0$ holds by taking the limit for (\ref{3}) above. In this way, we get $u^*\in S_D$.\\
Combining the two cases above, the conclusion holds.
\end{proof}

\begin{remark}
If $\lambda_{n}\to\lambda>0$ as $n\to\infty$, there exists $N_1$ such that $\frac{3\lambda}{2}\geq\lambda_{n}\geq\frac{\lambda}{2}>0$ for all $n>N_1$. Then from the fact that for any $z\in H$,
\begin{equation}\label{eq41}
  \min\{1,\bar{\lambda}\}\|e(z,1)\|\leq\|e(z,\bar{\lambda})\|\leq\max\{1,\bar{\lambda}\}\|e(z,1)\|,
\end{equation}
we can get
\begin{equation*}
  \|e(z_n,1)\|\leq\frac{\|e(z_n,\lambda_n)\|}{\min\{1,\frac{\lambda}{2}\}}=\frac{\|z_n-w_n\|}{\min\{1,\frac{\lambda}{2}\}}
  =\frac{\lambda_n}{\min\{1,\frac{\lambda}{2}\}}\frac{\|z_n-w_n\|}{\lambda_n}
  \leq\frac{3\lambda}{\min\{2,\lambda\}}\frac{\|z_n-w_n\|}{\lambda_n}.
\end{equation*}
If $\lambda_{n}\to0$ as $n\to\infty$, there exists $N_2$ such that $\lambda_{n}<1$ for all $n>N_2$. Hence, by Lemma \ref{mono}, we have
\begin{equation*}
  \|e(z_n,1)\|\leq\frac{\|e(z_n,\lambda_n)\|}{\lambda_{n}}=\frac{\|z_n-w_n\|}{\lambda_{n}}.
\end{equation*}
Indeed, for $\epsilon>0$, $\|e(z_n,1)\|<\epsilon$ implies that $z_n$ is an approximate solution of the variational inequality problem.
\end{remark}

\begin{remark}
If $\lambda_{n}\to\lambda>0$ as $n\to\infty$, then by the fact that $\lim\limits_{n\to\infty}\|z_n-w_n\|=0$, we have
\begin{equation*}
  \lim\limits_{n\to\infty}\frac{\|z_n-w_n\|}{\lambda_n}=0.
\end{equation*}
If $\lambda_{n}\to0$ as $n\to\infty$, then according to (\ref{555}) we can get for all $n_k\in I$,
\begin{equation*}
  \lim\limits_{n\to\infty}\frac{\|z_{n_k}-w_{n_k}\|}{\lambda_{n_k}}
  =0.
\end{equation*}
Thus, we obtain
\begin{equation*}
  \liminf\limits_{n\to\infty}\frac{\|z_n-w_n\|}{\lambda_n}=0.
\end{equation*}
This implies that for any $\epsilon>0$, there exists $n$, such that $\frac{\|z_n-w_n\|}{\lambda_n}<\epsilon$. Consequently, the algorithm can terminate at a finite step and $z_n$ can be viewed as an approximate solution of the variational inequality problem.\\
It is worth noting that we use $\frac{\|z_n - w_n\|}{\lambda_n} < \epsilon$ as the stopping criterion instead of $\|e(z_n, 1)\| < \epsilon$. This is because the latter needs to perform additional projection onto the set $C$, whereas the former does not.

\end{remark}

Next, we analyze the convergence behavior of the sequence $\{z_n\}$ produced by the Algorithm \ref{AL1} in finite dimensional space.
\begin{theorem}\label{th3.1}
Assume that conditions (A1)-(A4) are satisfied and $F(z)\neq 0$ for any $z\in S\setminus S_D$. The sequence $\{z_n\}$ generated by Algorithm \ref{AL1} converges to a point $z^\ell\in S$.
\end{theorem}
\begin{proof}
It follows from Lemma \ref{22} that $\{z_n\}$ is bounded, there exists a subsequence $\{z_{n_k}\}$ such that $z_{n_k}\to z^\ell$. \\
If $\lambda_{n}\to\lambda>0$ as $n\to\infty$, then we can get $\lim\limits_{k\to\infty}\lambda_{n_k}=\lambda>0$, which implies that there exists $K_1$ such that $\lambda_{n_k}\geq\frac{\lambda}{2}>0$ for all $k>K_1>K$. Then, from (\ref{eq41}) and (\ref{xy0}), we obtain
\begin{equation*}
  \|e(z^\ell,1)\|=\lim\limits_{k\to\infty}\|e(z_{n_k},1)\|\leq\lim\limits_{k\to\infty}\frac{\|z_{n_k}-w_{n_k}\|}{\min\{1,\frac{\lambda}{2}\}}=0.
\end{equation*}
If $\lambda_{n}\to0$ as $n\to\infty$, then clearly $I:=\{n\in\mathbb{N}: \lambda_{n+1}<\lambda_{n}\}$ is an infinite set. Let us denote the indices in set $I$ by $\{n_k\}$. It is obvious that $\lambda_{n_k}\to0$ as $k\to\infty$. The boundedness of $\{z_{n_k}\}$ ensures the existence of a convergent subsequence, which we still denote by $\{z_{n_k}\}$, with limit $z^\ell$, i.e., $z_{n_k}\to z^\ell$. Moreover, $\{\lambda_{n_k}\}$ satisfies $\lambda_{n_k + 1} < \lambda_{n_k}$. Hence, according to the definition of $\{\lambda_{n_k+1}\}$, we have
\begin{equation*}
  \lambda_{n_k+1}=\frac{\mu\|z_{n_k}-w_{n_k}\|}{\|F(z_{n_k})-F(w_{n_k})\|}<\lambda_{n_k}.
\end{equation*}
Since $\lim\limits_{k\to\infty}\lambda_{n_k}=0$, there exists $K_2$ such that $\lambda_{n_k}<1$ for all $k>K_2>K$.
Thus, by Lemma \ref{mono} and Lemma \ref{444} (i), it follows that
\begin{equation*}
  \|e(z^\ell,1)\|=\lim\limits_{k\to\infty}\|e(z_{n_k},1)\|\leq\lim\limits_{k\to\infty}\frac{\|z_{n_k}-w_{n_k}\|}{\lambda_{n_k}}
  =0.
\end{equation*}
This implies $z^\ell=P_C(z^\ell-F(z^\ell))$, that is $z^\ell\in S$. Then, it follows from the assumption that $F(z^\ell)\neq 0$ for all $z^\ell\in S\setminus S_D$, and from the proof of Lemma \ref{333}, we have $z^\ell\in S_D$. Furthermore, from Lemma \ref{22} (i), we obtain $\{\|z_n-z^\ell\|\}$ is convergent. Therefore, from $z_{n_k}\to z^\ell$, we obtain $z_n\to z^\ell$ as $n\to\infty$.
\end{proof}

\begin{remark}
In fact, the proof of Lemma \ref{333} already implies that $z^\ell\in S$. Here, we present another approach to establish that $z^\ell\in S$.
\end{remark}

In the following, we establish the strong convergence of the sequence $\{z_n\}$ generated by the algorithm in infinite dimensional Hilbert spaces under the following condition:
\begin{itemize}
  \item \emph{(A5)} For $\{w_n\}\subseteq C$, if $w_n\rightharpoonup u\in S_D$, then there exist $\epsilon\geq0$ such that the following inequality holds:
\begin{equation}\label{A5}
  \liminf\limits_{n\to\infty}\frac{\langle F(w_n), w_n - u \rangle}{\|w_n - u\|^{2 + \epsilon}} > 0.
\end{equation}
\end{itemize}

\begin{remark}
For $\{w_n\}\subseteq C$ and $w_n\rightharpoonup u\in S_D\subseteq S$, we can obtain $\langle F(u), w_n-u \rangle\geq 0$. If $F$ is $\eta$-strongly pseudomonotone operator, it is clear that $\langle F(w_n), w_n-u \rangle\geq \eta\|w_n-u\|^2$. This implies that (\ref{A5}) holds with $\epsilon=0$.


\end{remark}

\begin{theorem}\label{th3.1}
Assume that conditions (A1)-(A5) are satisfied and $F(z)\neq 0$ for any $z\in S\setminus S_D$. The sequence $\{z_n\}$ generated by Algorithm \ref{AL1} converges strongly to a point $z^\ell\in S$.
\end{theorem}
\begin{proof}
We still discuss the proof in two cases as follows:\\
\textbf{Case 1} $\lambda_{n}\to\lambda>0$ as $n\to\infty$.\\
According to Lemma \ref{333}, there exists the subsequence $\{z_{n_k}\}$ of $\{z_n\}$ such that $z_{n_k}\rightharpoonup z^\ell\in S$. And $F(z^\ell)\neq 0$ for all  $z^\ell\in S\setminus S_D$ means $z^\ell\in S_D$. Since $\lim\limits_{n\to\infty}\|z_{n_k}-w_{n_k}\|=0$, then $w_{n_k}\rightharpoonup z^\ell\in S_D$. And as $\{w_{n_k}\}\subseteq C$, by assumption \emph{(A5)}, we have $\exists K'>K$ and $c>0$ such that for all $k>K'$,
\begin{equation}\label{tiaojian}
  \langle F(w_{n_k}), w_{n_k}-z^\ell \rangle\geq c\|w_{n_k}-z^\ell\|^{2+\epsilon}.
\end{equation}
Then, by (\ref{3.331}), the definition of $\{\lambda_{{n_k}+1}\}$ and (\ref{tiaojian}), we have
\begin{align*}
  \|z_{{n_k}+1}-z^\ell\|^2
  \leq&\|z_{n_k}-z^\ell\|^2-\|w_{n_k}-z_{n_k}\|^2+\lambda_{n_k}^2\|F(w_{n_k})-F(z_{n_k})\|^2
  \notag\\&-2\lambda_{n_k}\langle F(w_{n_k}), w_{n_k}-z^\ell \rangle\notag\\
  \leq&\|z_{n_k}-z^\ell\|^2-(1-\frac{\mu^2\lambda_{n_k}^2}{\lambda_{{n_k}+1}^2})\|w_{n_k}-z_{n_k}\|^2
  -2\lambda_{n_k}\langle F(w_{n_k}), w_{n_k}-z^\ell \rangle\notag\\
  \leq&\|z_{n_k}-z^\ell\|^2-(1-\frac{\mu^2\lambda_{n_k}^2}{\lambda_{{n_k}+1}^2})\|w_{n_k}-z_{n_k}\|^2
  -2c\lambda_{n_k}\|w_{n_k}-z^\ell\|^{2+\epsilon}.
\end{align*}
Therefore,
\begin{equation}\label{345}
  2c\lambda_{n_k}\|w_{n_k}-z^\ell\|^{2+\epsilon}\leq\|z_{n_k}-z^\ell\|^2-\|z_{{n_k}+1}-z^\ell\|^2
  -(1-\frac{\mu^2\lambda_{n_k}^2}{\lambda_{{n_k}+1}^2})\|w_{n_k}-z_{n_k}\|^2.
\end{equation}
By Lemma \ref{22} (ii) and (iii), it is easy to know that the $\lim\limits_{n\to\infty}\|z_n-z^\ell\|$ exists and $\lim\limits_{k\to\infty}\|z_{n_k}-w_{n_k}\|=0$. Moreover, $\lim\limits_{k\to\infty}(1-\frac{\mu^2\lambda_{n_k}^2}{\lambda_{{n_k}+1}^2})= 1-\mu^2>0$. Then, from (\ref{345}), we conclude that $\lim\limits_{k\to\infty}\|w_{n_k}-z^\ell\|=0$. Hence,
\begin{equation*}
  \|z_{n_k}-z^\ell\|\leq\|z_{n_k}-w_{n_k}\|+\|w_{n_k}-z^\ell\|\to0,\ k\to\infty.
\end{equation*}
Consequently, by the fact that $\lim\limits_{n\to\infty}\|z_n-z^\ell\|$ exists, we obtain
\begin{equation*}
  \|z_n-z^\ell\|\to0, \ n\to\infty.
\end{equation*}
\textbf{Case 2} $\lambda_{n}\to0$ as $n\to\infty$.\\
If $\lambda_{n}\to0$ as $n\to\infty$, we know that $I:=\{n\in\mathbb{N}: \lambda_{n+1}<\lambda_{n}\}$ is an infinite set. And we denote the indices in set $I$ by $\{n_k\}$. It is obvious that $\lambda_{n_k}\to0$ as $k\to\infty$ and
\begin{equation*}
  \lambda_{{n_k}+1}<\lambda_{n_k}.
\end{equation*}
Furthermore, by the proof of Lemma \ref{333} and the assumption $F(z^\ell)\neq 0$ for all  $z^\ell\in S\setminus S_D$, we can get $z_{n_k}\rightharpoonup z^\ell\in S_D$. Then, by (\ref{xy0}), we have $w_{n_k}\rightharpoonup z^\ell\in S_D$. Thus by assumption \emph{(A5)}, it follows that (\ref{tiaojian}) holds.\\
According to (\ref{3.331}), (\ref{bb}) and (\ref{tiaojian}), for sufficiently large $k$, we get
\begin{align*}
  \|z_{{n_k}+1}-z^\ell\|^2\leq&\|z_{n_k}-z^\ell\|^2-(1-\mu^2)\|z_{n_k}-w_{n_k}\|^2-2\lambda_{n_k}\langle F(w_{n_k}), w_{n_k}-z^\ell \rangle\\
  &+(\lambda_{n_k}^2-\lambda_{{n_k}+1}^2)\\
  \leq&\|z_{n_k}-z^\ell\|^2-2\lambda_{n_k}\langle F(w_{n_k}), w_{n_k}-z^\ell \rangle+(\lambda_{n_k}^2-\lambda_{{n_k}+1}^2)\\
  \leq&\|z_{n_k}-z^\ell\|^2-2c\lambda_{n_k}\|w_{n_k}-z^\ell\|^{2+\epsilon}+\lambda_{n_k}^2.
\end{align*}
Since $\{z_{n_k}\}$ is bounded, for any $k$, there exists $M_1$ such that $\|z_{n_k}-z^\ell\|+\|z_{{n_k}+1}-z^\ell\|\leq M_1$. Hence,
\begin{align*}
  2c\|w_{n_k}-z^\ell\|^{2+\epsilon}\leq&\frac{1}{\lambda_{n_k}}(\|z_{n_k}-z^\ell\|^2-\|z_{{n_k}+1}-z^\ell\|^2)+\lambda_{n_k}\\
  =&\frac{1}{\lambda_{n_k}}[(\|z_{n_k}-z^\ell\|-\|z_{{n_k}+1}-z^\ell\|)(\|z_{n_k}-z^\ell\|+\|z_{{n_k}+1}-z^\ell\|)]
  +\lambda_{n_k}\\
  \leq&M_1\frac{\|z_{n_k}-z_{{n_k}+1}\|}{\lambda_{n_k}}+\lambda_{n_k}.
\end{align*}
Then, from $\lim\limits_{k\to\infty}\lambda_{n_k}=0$ and Lemma \ref{444} (ii), we have
\begin{equation*}
  \lim\limits_{k\to\infty}\|w_{n_k}-z^\ell\|=0.
\end{equation*}
Thus,
\begin{equation*}
  \|z_{n_k}-z^\ell\|\leq\|z_{n_k}-w_{n_k}\|+\|w_{n_k}-z^\ell\|\to0,\ k\to\infty.
\end{equation*}
Moreover, according to Lemma \ref{22} (ii) it is known that $\lim\limits_{n\to\infty}\|z_n-z^\ell\|$ exists. Therefore, we finally get
\begin{equation*}
  \lim\limits_{n\to\infty}\|z_n-z^\ell\|=0.
\end{equation*}
This completes the proof.
\end{proof}

\section{Numerical Experiments}\label{sec55}
This section aims to validate the effectiveness and advantages of the proposed algorithm by comparing it with Algorithm 1 and Algorithm 2 presented in \cite{thong201911}. In the following tables, ``Iter'' indicates the number of iterations performed, and ``CPU Time'' measures runtime in seconds. All codes were developed in MATLAB R2022a and tested on a PC with an Intel(R) Core(TM) i7-6700 CPU @ 3.40 GHz 3.41 GHz and 16.00 GB of RAM.\\
\textbf{Example 4.1.} Let $H=l_2$ and $C:=\{z=(z_1, z_2,\ldots, z_j,\ldots)\in H: |z_j|\leq \frac{1}{j}, j=1,2,\ldots,n,\ldots\}$. Consider the operator $F : C \to H$ defined as follows: for some $\theta>0$,
\begin{equation*}
  F(z)=(\|z\|+\frac{1}{\|z\|+\theta})z.
\end{equation*}
It follows directly that $S=\{0\}$. In addition, the operator $F$ is pseudomonotone on $H$, and it is uniformly continuous and sequentially weakly continuous on $C$, however, it does not satisfy Lipschitz continuity on $H$ (refer to \cite{thong201911} for more details). We consider $H=\mathbb{R}^m$ with varying dimensions, and different $\theta$. In this case, the feasible set $C$ takes the following form:
\begin{equation*}
  C:=\{x\in \mathbb{R}^m: -\frac{1}{j}\leq x_j\leq \frac{1}{j}, j=1,2,\ldots,m\}.
\end{equation*}
The parameters are set as follows:
\begin{itemize}
  \item For our proposed Algorithm \ref{AL1}, we set $\mu=0.3$, $\lambda_1=0.01$ and $\xi_n=\frac{1}{(n+1)^{1.1}}$.
  \item For Algorithm 1 and Algorithm 2 in \cite{thong201911}, let $\gamma=0.1$, $l=0.5$, $\mu=0.8$. In addition, for all $n\in \mathbb{N}$,
      \begin{equation*}
        \alpha_n=\frac{1}{\sqrt{n}+2},\ \beta_n=\frac{1-\alpha_n}{2}.
      \end{equation*}
\end{itemize}
The initial point is set as $x_1 = (1, 1, \ldots, 1) \in \mathbb{R}^m$. The stopping criterion is defined by $E_n = \frac{\|z_n - w_n\|}{\lambda_n} < 10^{-8}$, with the maximum number of iterations limited to 5000.\\
We compare the proposed Algorithm \ref{AL1} (referred to as Alg. \ref{AL1}) with Algorithm 1 and Algorithm 2 from \cite{thong201911} (referred to as Alg. 1 and Alg. 2, respectively) under various dimensions $m$ and different values of $\theta$. The evaluation focused on two key metrics: the error measure $E_n$ and CPU time. To ensure statistical reliability, all CPU time measurements were obtained by averaging the results from 20 randomized runs. The comparative results are visually presented in Fig. \ref{fig1} and Fig. \ref{fig2}. Additionally, we compare all algorithms in terms of the number of iterations, the number of projection operations (denoted as Num$_{P_C}$), the number of function evaluations (denoted as Num$_F$), and CPU time. The results are presented in Table \ref{table1} and Table \ref{table2}. which clearly demonstrate the effectiveness and superiority of our proposed algorithm.

\begin{figure}[htbp]
    \centering
    \subfigure[$\theta=1$]{%
        \includegraphics[width=0.3\textwidth]{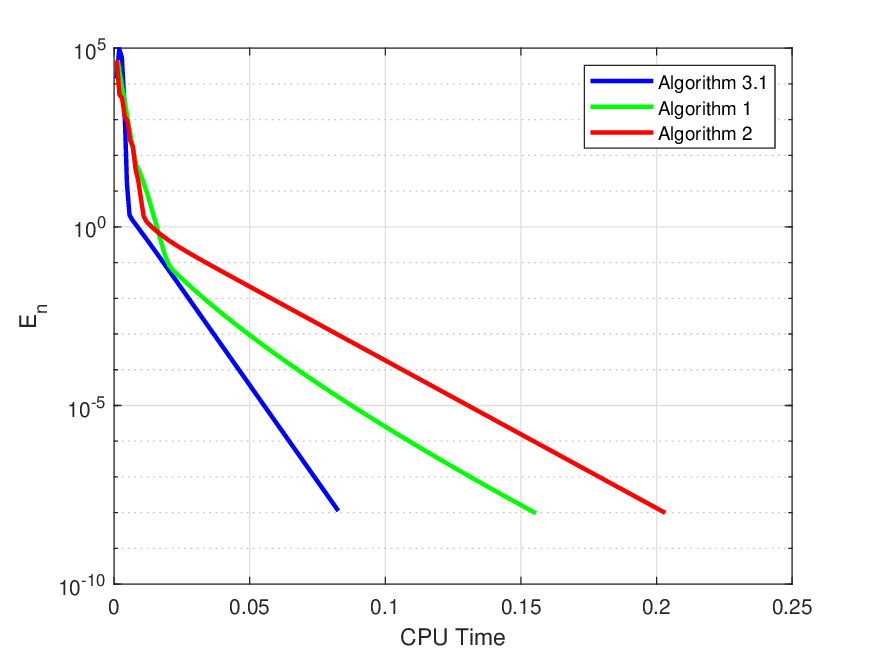}
        \label{fig:1}
    }
    \hfill
    \subfigure[$\theta=5$]{%
        \includegraphics[width=0.3\textwidth]{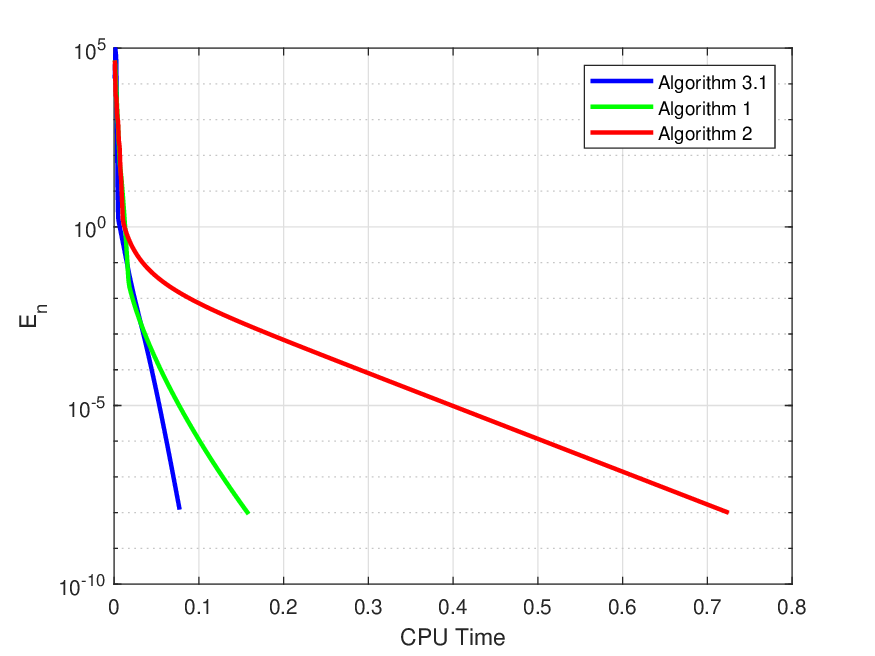}
        \label{fig:2}
    }
    \hfill
    \subfigure[$\theta=10$]{%
        \includegraphics[width=0.3\textwidth]{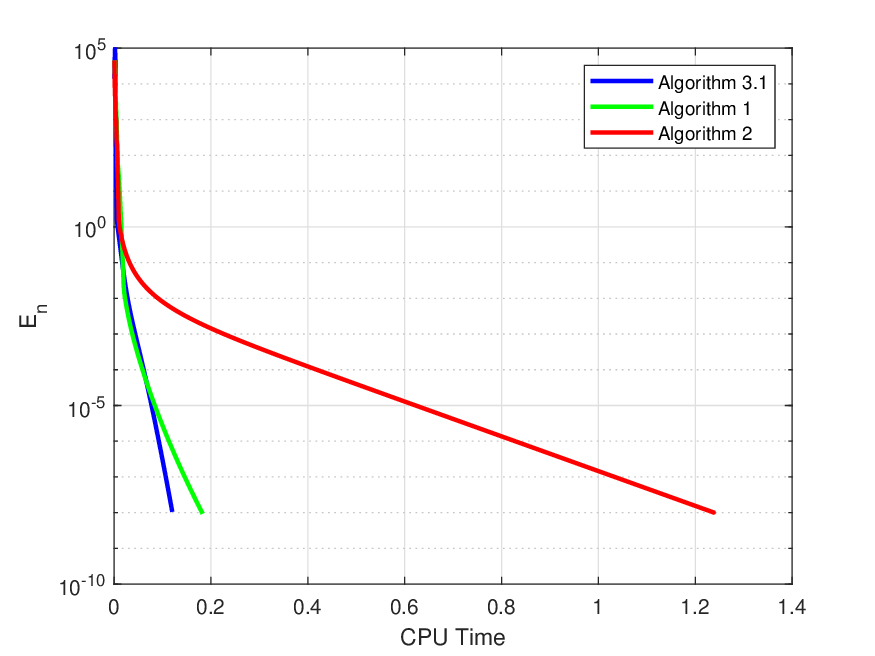}
        \label{fig:3}
    }

    \caption{Comparison of $E_n$ and CPU Time for different values of $\theta$ with $m=20000$.}
    \label{fig1}
\end{figure}

\begin{figure}[htbp]
    \centering
    \subfigure[$\theta=1$]{%
        \includegraphics[width=0.3\textwidth]{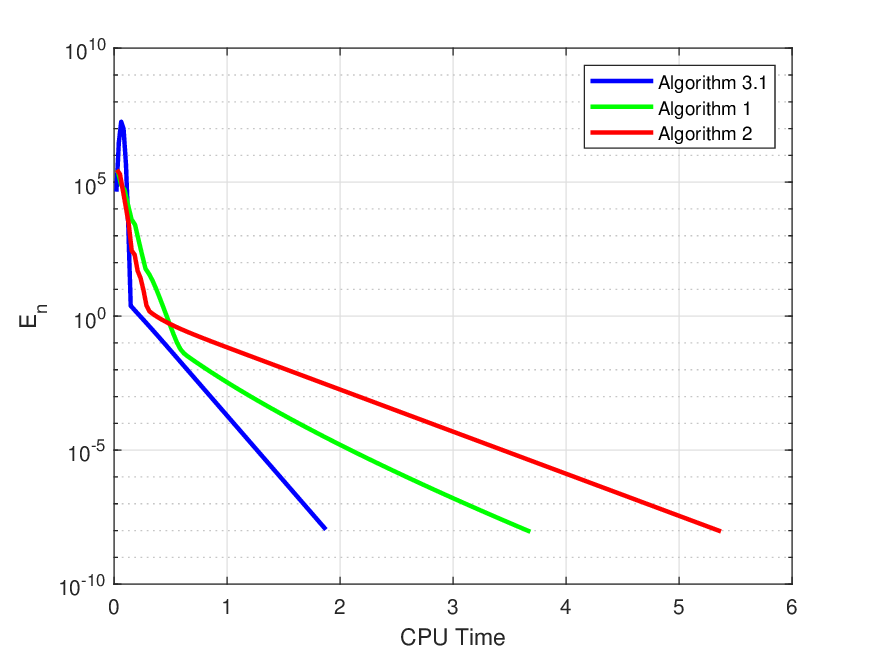}
        \label{fig:11}
    }
    \hfill
    \subfigure[$\theta=5$]{%
        \includegraphics[width=0.3\textwidth]{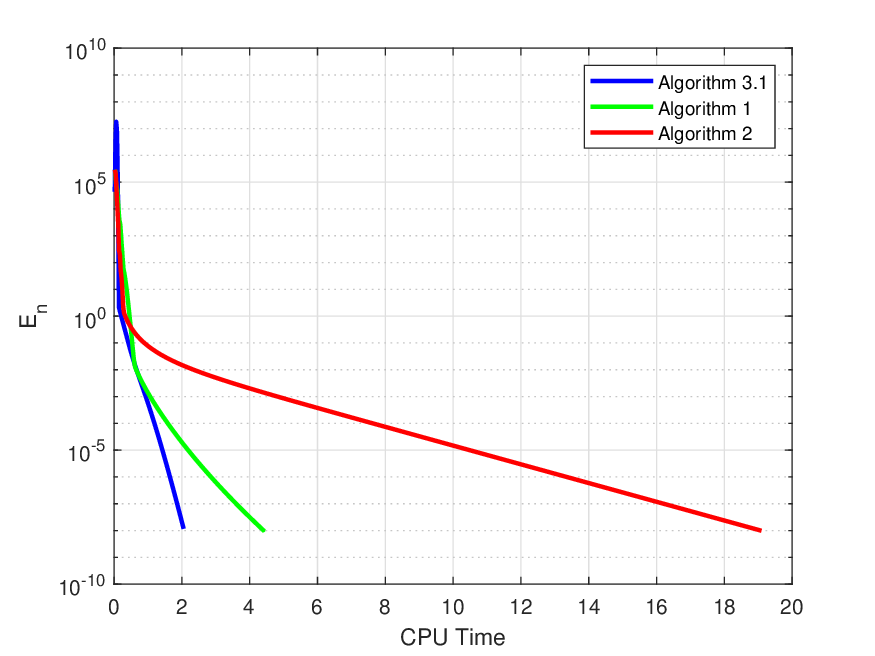}
        \label{fig:22}
    }
    \hfill
    \subfigure[$\theta=10$]{%
        \includegraphics[width=0.3\textwidth]{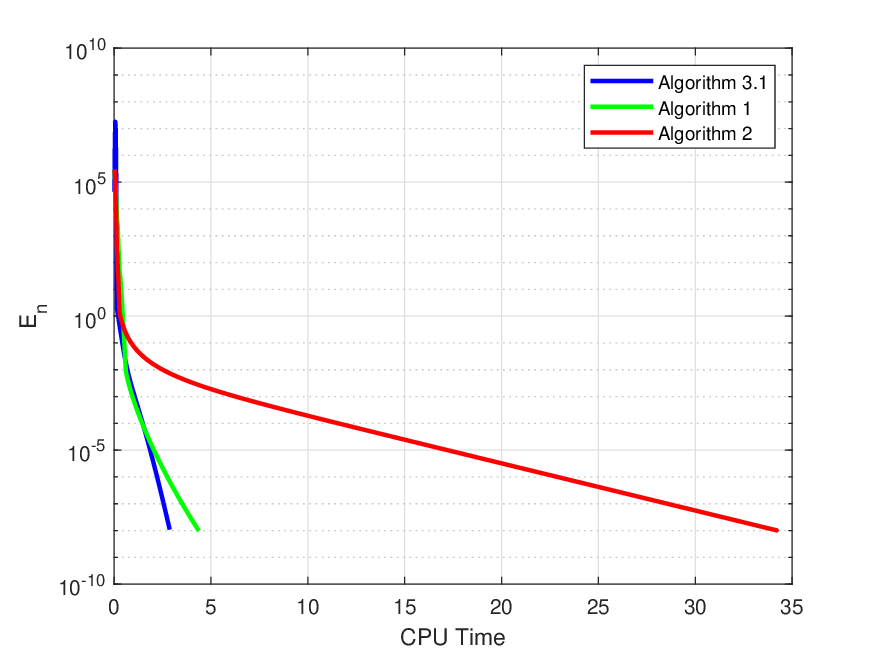}
        \label{fig:33}
    }

    \caption{Comparison of $E_n$ and CPU Time for different values of $\theta$ with $m=200000$.}
    \label{fig2}
\end{figure}

\begin{table}[htpb]
\caption{Numerical performance of all algorithms with $m=20000$.}
\begin{tabular}{ccccccccc}
\toprule
$\theta$&Algorithm&Iter&Num$_{P_C}$&Num$_F$&CPU time\\
\midrule
$\theta=1$&Alg. \ref{AL1}&88&88&176&0.0836\\
&Alg. 1&205&222&427&0.2031\\
&Alg. 2&115&130&245&0.1555\\
\midrule
$\theta=5$&Alg. \ref{AL1}&96&96&192&0.0781\\
&Alg. 1&775&792&1567&0.7250\\
&Alg. 2&140&155&295&0.1586\\
\midrule
$\theta=10$&Alg. \ref{AL1}&132&132&264&0.1211\\
&Alg. 1&1396&1413&2809&1.2398\\
&Alg. 2&139&154&293&0.1828\\
\bottomrule
\end{tabular}
\label{table1}
\end{table}

\begin{table}[htpb]
\caption{Numerical performance of all algorithms with $m=200000$.}
\begin{tabular}{ccccccccc}
\toprule
$\theta$&Algorithm&Iter&Num$_{P_C}$&Num$_F$&CPU time\\
\midrule
$\theta=1$&Alg. \ref{AL1}&90&90&180&1.8945\\
&Alg. 1&206&232&438&5.3688\\
&Alg. 2&119&148&267&3.6820\\
\midrule
$\theta=5$&Alg. \ref{AL1}&100&100&200&2.0727\\
&Alg.1&776&802&1578&19.0930\\
&Alg.2&144&173&317&4.4313\\
\midrule
$\theta=10$&Alg. \ref{AL1}&137&137&274&2.8922\\
&Alg.1&1396&1422&2818&34.2523\\
&Alg.2&143&172&315&4.3781\\
\bottomrule
\end{tabular}
\label{table2}
\end{table}

\section{Conclusion}\label{sec6}
This work adopts Tseng's extragradient method combined with the self-adaptive stepsize criterion to solve variational inequality problems characterized by quasimonotone and non-Lipschitz operators in real Hilbert spaces. We establish the convergence of the iterative sequence in finite dimensional spaces and, under suitable assumptions, further demonstrate its strong convergence to a solution of the problem in infinite dimensional Hilbert spaces. Finally, we compare the proposed self-adaptive algorithm with existing linesearch-based methods in the literature. The results provide clear evidence of the enhanced performance achieved by the proposed method.

\bmhead{Acknowledgements}
The authors express their appreciation to the responsible editor and anonymous reviewers for their careful reading.

\bmhead{Funding} This work was supported by the National Natural Science Foundation of China (No. 12261019), the Natural Science Basic Research Program Project of Shaanxi Province (No. 2024JC-YBMS-019), the Natural Science Foundation of Shaanxi Province of China (No. 2023-JC-YB-049), the Fundamental Research Funds for the Central Universities, and the Innovation Fund of Xidian University (No. YJSJ25009).
\section*{Data Availability} No data were generated or analyzed in this theoretical study, therefore data availability is not applicable.
\section*{Declarations}

\textbf{Conflict of interest} The authors declare no conflicts of interest.

\bigskip
\bibliography{sn-bibliography}

\end{document}